\documentclass{amsart}

\theoremstyle{plain}
\newtheorem{theorem}{Theorem}

\theoremstyle{definition}

\theoremstyle{remark}

\newtheorem{example}[theorem]{Example}

\providecommand{\abs}[1]{\lvert#1\rvert}
\providecommand{\Abs}[1]{\Bigl\lvert#1\Bigr\rvert}

\usepackage{amssymb}
\usepackage{xcolor}
\usepackage{graphicx}
\usepackage{wrapfig,lipsum}
\usepackage{multirow}
\usepackage{subcaption}
\usepackage{afterpage}
\usepackage{float}
\usepackage{algorithm2e}

\usepackage{bbm}

\begin{document}

\title[Asymptotics of predictive distributions]{Weak convergence of predictive distributions}
\author{Fabrizio Leisen}
\address{Fabrizio Leisen, Department of Mathematics, king’s College, Strand WC2R 2LS, London, Uk}
\email{fabrizio.leisen@gmail.com}
\author{Luca Pratelli}
\address{Luca Pratelli, Accademia Navale, viale Italia 72, 57100 Livorno,
Italy} \email{luca{\_}pratelli@marina.difesa.it}
\author{Pietro Rigo}
\address{Pietro Rigo (corresponding author), Dipartimento di Scienze Statistiche ``P. Fortunati'', Universit\`a di Bologna, via delle Belle Arti 41, 40126 Bologna, Italy}
\email{pietro.rigo@unibo.it}
\keywords{Asymptotic exchangeability, Conditional identity in distribution, Predictive distribution, Random probability measure, Stable convergence}
\subjclass[2020]{60B10, 60G57, 60G09, 60F99}

\begin{abstract}
Let $(X_n)$ be a sequence of random variables with values in a standard Borel space $S$. We investigate the condition
\begin{gather}\label{x56w1q}
E\bigl\{f(X_{n+1})\mid X_1,\ldots,X_n\bigr\}\,\quad\text{converges in probability,}\tag{*}
\\\text{as }n\rightarrow\infty,\text{ for each bounded Borel function }f:S\rightarrow\mathbb{R}.\notag
\end{gather}
Some consequences of \eqref{x56w1q} are highlighted and various sufficient conditions for it are obtained. In particular, \eqref{x56w1q} is characterized in terms of stable convergence. Since \eqref{x56w1q} holds whenever $(X_n)$ is conditionally identically distributed, three weak versions of the latter condition are investigated as well. For each of such versions, our main goal is proving (or disproving) that \eqref{x56w1q} holds. Several counterexamples are given.
\end{abstract}

\maketitle

\section{introduction}\label{intro}

Throughout, $(S,\mathcal{B})$ is a standard Borel space, that is, $S$ is a Borel subset of some Polish space and $\mathcal{B}$ the Borel $\sigma$-field on $S$. We denote by $M_b(S)$ the set of real bounded $\mathcal{B}$-measurable functions on $S$. Moreover, $(\Omega,\mathcal{A},P)$ is a probability space,
$$X=(X_1,X_2,\ldots)$$
a sequence of $S$-valued random variables on $(\Omega,\mathcal{A},P)$, and
$$\mathcal{F}_0=\bigl\{\emptyset,\Omega\bigr\}\quad\text{and}\quad\mathcal{F}_n=\sigma(X_1,\ldots,X_n).$$

A {\em random probability measure} (r.p.m.) on $\mathcal{B}$ is a map $\alpha$ on $\Omega\times\mathcal{B}$ such that:

$-$ $\alpha(\omega,\cdot)$ is a probability measure on $\mathcal{B}$ for fixed $\omega\in\Omega$;

$-$ the function $\omega\mapsto\alpha(\omega,B)$ is $\mathcal{A}$-measurable for fixed $B\in\mathcal{B}$.

\noindent If $f\in M_b(S)$, we denote by $\alpha(f)$ the real random variable $\omega\mapsto\alpha(\omega,f)$ where
$$\alpha(\omega,f)=\int f(x)\,\alpha(\omega,dx).$$
However, we write $\alpha(B)$ instead of $\alpha(1_B)$ in case $f=1_B$ with $B\in\mathcal{B}$.

For $n\ge 0$, define
$$\alpha_n(\cdot)=P(X_{n+1}\in\cdot\mid\mathcal{F}_n).$$
Such $\alpha_n$ is a r.p.m. on $\mathcal{B}$ such that
$$\alpha_n(f)=E\bigl\{f(X_{n+1})\mid\mathcal{F}_n\bigr\}\quad\quad\text{a.s. for each }f\in M_b(S).$$
The $\alpha_n$ are usually called the {\em predictive distributions} of $X$.

In various frameworks, including Bayesian inference, empirical processes and species sampling, a (natural) question is whether $\alpha_n$ converges, in some sense, as $n\rightarrow\infty$; see e.g. \cite{BPREJP,BPR13,BERN21,DV,FHW23,FY,HMW,MF,MOW,PIT} and references therein. Thus, in this paper, we investigate the condition
\begin{gather}\label{c1}
\alpha_n(f)\text{ converges in probability, as }n\rightarrow\infty,\text{ for each }f\in M_b(S).
\end{gather}
Indeed, in addition to its possible theoretical interest, condition \eqref{c1} has some useful consequences; see Remarks (i)-(v) below.

A popular sufficient condition for \eqref{c1} is that $X$ is {\em conditionally identically distributed} (c.i.d.), in the sense that
\begin{gather*}
(X_1,\ldots,X_n,X_k)\sim (X_1,\ldots,X_n,X_{n+1})\quad\quad\text{for all }k>n\ge 0.
\end{gather*}
In fact, as shown in \cite{BPRCID}, if $X$ is c.i.d. then
\begin{gather}\label{c2}
\alpha_n(f)\text{ converges a.s., as }n\rightarrow\infty,\text{ for each }f\in M_b(S).\tag{1*}
\end{gather}
Condition \eqref{c2} is strictly stronger than \eqref{c1}; see Example \ref{v778m4w}. Anyway, since
$$X\text{ c.i.d.}\quad\Rightarrow\quad\eqref{c2}\quad\Rightarrow\quad\eqref{c1},$$
we also investigate some weak versions of the c.i.d. condition. For each of such versions, our main goal is proving (or disproving) that \eqref{c1} or \eqref{c2} hold. However, apart from \eqref{c1} and \eqref{c2}, these versions could be potentially interesting in themselves.

Our main results are a characterization of \eqref{c1} in terms of stable convergence and various sufficient (but not necessary) conditions. Importantly, we also give several examples highlighting the connections between \eqref{c1} and other related conditions.

Finally, to motivate condition \eqref{c1}, we list some of its consequences.

\medskip

\textbf{(i)} Under \eqref{c1}, the limit of $\alpha_n(f)$ can be written as $\alpha(f)$ for a single r.p.m. $\alpha$. In fact, by Theorem \ref{bgy7m} below, since $(S,\mathcal{B})$ is standard Borel, condition \eqref{c1} yields
\begin{gather*}
\text{There is a r.p.m. }\alpha\text{ on }\mathcal{B}\text{ such that }
\alpha_n(f)\overset{P}\longrightarrow\alpha(f)\,\text{ for each }f\in M_b(S).\notag
\end{gather*}

\medskip

\textbf{(ii)} Let $\beta_n$ and $\beta$ be r.p.m.'s on $\mathcal{B}$. Write $\beta_n\overset{P}\longrightarrow\beta$ to mean that, for each subsequence $(n_j)$, there is a sub-subsequence $(n_{j_k})$ of $(n_j)$ such that
$$\beta_{n_{j_k}}(\omega,\cdot)\overset{weakly}\longrightarrow\beta(\omega,\cdot),\quad\text{as }k\rightarrow\infty,\text{ for almost all }\omega\in\Omega.$$
Equivalently, $\beta_n\overset{P}\longrightarrow\beta$ if and only if $d(\beta_n,\beta)\overset{P}\longrightarrow 0$ where $d$ is any distance metrizing weak convergence of probability measures (such as the Prohorov distance or the bounded Lipschitz metric). Then, by \cite[Cor. 2.4]{BPRSTOC}, condition \eqref{c1} yields $\alpha_n\overset{P}\longrightarrow\alpha$ for some r.p.m. $\alpha$. Incidentally, this explains the title of this paper.

\medskip

\textbf{(iii)} Let $\mu_n$ be the {\em empirical measure}, i.e., $\mu_n=(1/n)\,\sum_{i=1}^n\delta_{X_i}$ where $\delta_x$ denotes the unit mass at the point $x\in S$. Say that $X$ {\em satisfies the weak law of large numbers} if $\mu_n(f)=(1/n)\,\sum_{i=1}^{n}f(X_i)$ converges in probability for each $f\in M_b(S)$. By a (well known) martingale argument, one obtains
$$\mu_n(f)-\frac{1}{n}\,\sum_{i=1}^{n}\alpha_{i-1}(f)=\frac{1}{n}\,\sum_{i=1}^{n}\Bigl\{f(X_i)-\alpha_{i-1}(f)\Bigr\}\overset{a.s.}\longrightarrow 0\quad\text{for each }f\in M_b(S).$$
Hence, under condition \eqref{c1}, $X$ satisfies the weak law of large numbers. Precisely, $\mu_n(f)\overset{P}\longrightarrow\alpha(f)$ for all $f\in M_b(S)$ where $\alpha$ is the r.p.m. involved in Remark (i).

\medskip

\textbf{(iv)} By Remark (iii), condition \eqref{c1} implies $\mu_n(f)-\alpha_n(f)\overset{P}\longrightarrow 0$ for each $f\in M_b(S)$. Hence, for large $n$, the empirical measure $\mu_n$ is a reasonable approximation of the predictive $\alpha_n$. Using a statistical language, under condition \eqref{c1}, $\mu_n$ is a consistent estimate of $\alpha_n$.

\medskip

\textbf{(v)} Condition \eqref{c1} implies that $X$ is {\em asymptotically exchangeable}, that is
$$(X_{n+1},X_{n+2},\ldots)\longrightarrow (Z_1,Z_2,\ldots)\quad\text{in distribution, as }n\rightarrow\infty,$$
where $(Z_1,Z_2,\ldots)$ is an exchangeable sequence. This is proved in forthcoming Theorem \ref{v67j9x2}, which slightly improves a result by Aldous; see Lemma (8.2) of \cite{ALDOUS}.

\section{An analysis of condition \eqref{c1}}\label{x45g7n}
In the sequel, $C_b(S)$ denotes the set of real bounded continuous functions on $S$. We also recall that $(S,\mathcal{B})$ is a standard Borel space. This assumption plays a role quite often, for instance in Theorem \ref{bgy7m} below.

We begin with Remark (i). Given a sub-$\sigma$-field $\mathcal{G}\subset\mathcal{A}$, a r.p.m. $\alpha$ is said to be $\mathcal{G}$-measurable if the function $\omega\mapsto\alpha(\omega,B)$ is $\mathcal{G}$-measurable for fixed $B\in\mathcal{B}$.

\begin{theorem}\label{bgy7m}
Let $\mathcal{T}=\bigcap_n\sigma(X_n,X_{n+1},\ldots)$ be the tail $\sigma$-field of $X$. If condition \eqref{c1} holds, there is a r.p.m. $\alpha$ on $\mathcal{B}$ such that $\alpha$ is $\mathcal{T}$-measurable and $\alpha_n(f)\overset{P}\longrightarrow\alpha(f)$ for each $f\in M_b(S)$.
\end{theorem}

\begin{proof}
By Corollary 2.4 of \cite{BPRSTOC}, since $(S,\mathcal{B})$ is standard Borel and $\alpha_n(f)$ converges in probability whenever $f\in C_b(S)$, there is a r.p.m. $\alpha$ on $\mathcal{B}$ such that $\alpha_n(f)\overset{P}\longrightarrow\alpha(f)$ for each $f\in C_b(S)$. By Remark (iii), one also obtains $\mu_n(f)\overset{P}\longrightarrow\alpha(f)$ for each $f\in C_b(S)$, where $\mu_n=(1/n)\,\sum_{i=1}^n\delta_{X_i}$ is the empirical measure. Hence, along a suitable subsequence $(n_j)$, one obtains
$$\mu_{n_j}(\omega,\cdot)\overset{weakly}\longrightarrow\alpha(\omega,\cdot),\text{ as }j\rightarrow\infty,\text{ for almost all }\omega\in\Omega.$$
Therefore, $\alpha$ can be taken to be $\mathcal{T}$-measurable. Next, for each $f\in M_b(S)$, fix a random variable $\gamma(f)$ satisfying $\alpha_n(f)\overset{P}\longrightarrow\gamma(f)$ and define
$$L=\bigl\{f\in M_b(S):\alpha(f)=\gamma(f)\text{ a.s.}\bigr\}.$$
Then, $L\supset C_b(S)$ and $L$ is a linear space including the constants. Hence, it suffices to show that
\begin{gather}\label{moncond}
f\in L\text{ provided }f\in M_b(S)\text{ and }f_k\rightarrow f\text{ pointwise}
\\\text{for some monotone sequence }(f_k)\subset L.\notag
\end{gather}
In fact, if $L$ satisfies \eqref{moncond}, the monotone class theorem yields $L=M_b(S)$.

We next prove \eqref{moncond}. Fix $f\in M_b(S)$ and a monotone sequence $(f_k)\subset L$ such that $f_k\rightarrow f$ pointwise. Since $L$ is a linear space, $(f_k)$ can be assumed to be increasing, i.e. $f_1\le f_2\le\ldots$ In this case,
$$\alpha(f)=\sup_k\alpha(f_k)=\sup_k\gamma(f_k)\le\gamma(f)\quad\quad\text{a.s.}$$
Next, condition \eqref{c1} implies that $\lim_nE\bigl\{\alpha_n(B)\bigr\}$ exists for every $B\in\mathcal{B}$. Define
$$\lambda(B)=\lim_nE\bigl\{\alpha_n(B)\bigr\}\quad\text{and}\quad\lambda^*(B)=E\bigl\{\alpha(B)\bigr\}\quad\text{for all }B\in\mathcal{B}.$$
By the Vitali-Hahn-Saks theorem, $\lambda$ is a probability measure on $\mathcal{B}$. Moreover,
$$\lambda(g)=\lim_nE\bigl\{\alpha_n(g)\bigr\}=E\bigl\{\gamma(g)\bigr\}=E\bigl\{\alpha(g)\bigr\}=\lambda^*(g)\quad\text{for all }g\in C_b(S).$$
Therefore, $\lambda=\lambda^*$ and this in turn implies
\begin{gather*}
E\bigl\{\gamma(f)-\alpha(f)\bigr\}=E\bigl\{\gamma(f)\bigr\}-E\bigl\{\alpha(f)\bigr\}
\\\le\lim_n E\bigl\{\alpha_n(f)\bigr\}-\lambda(f)=\lambda(f)-\lambda(f)=0,
\end{gather*}
where the inequality is due to Fatou's lemma. Since $\gamma(f)-\alpha(f)\ge 0$ a.s., it follows that $\alpha(f)=\gamma(f)$ a.s. Hence, $f\in L$ and this concludes the proof.
\end{proof}

An obvious weakening of condition \eqref{c1} is
\begin{gather}\label{c55b9m}
\alpha_n(f)\text{ converges in probability, as }n\rightarrow\infty,\text{ for each }f\in C_b(S).
\end{gather}
Condition \eqref{c55b9m} arises quite frequently. For instance, by \cite[Cor. 2.4]{BPRSTOC}, condition \eqref{c55b9m} amounts to $\alpha_n\overset{P}\longrightarrow\alpha$ for some r.p.m. $\alpha$; see Remark (ii). Or else, by the same argument of Theorem \ref{bgy7m}, it can be shown that

$$\eqref{c1}\quad\Longleftrightarrow\quad\eqref{c55b9m}\text{ and }\eqref{x55hgf6j}$$
where condition \eqref{x55hgf6j} is
\begin{gather}\label{x55hgf6j}
\lim_nP(X_n\in B)\,\text{ exists for every }B\in\mathcal{B}.
\end{gather}
Note that condition \eqref{x55hgf6j} is trivially true if the sequence $X$ is identically distributed. Finally, condition \eqref{c55b9m} implies asymptotic exchangeability, as defined in Remark (v). The next result slightly improves Lemma (8.2) of Aldous \cite{ALDOUS}.

\begin{theorem}\label{v67j9x2}
Condition \eqref{c55b9m} implies asymptotic exchangeability of $X$.
\end{theorem}

The proof of Theorem \ref{v67j9x2} is postponed to the end of this section. Here, we note that Theorem \ref{v67j9x2} can not be inverted, i.e., asymptotic exchangeability does not imply condition \eqref{c55b9m}; see Example \ref{w34b7j}. A further remark is that, since exchangeability amounts to stationarity and asymptotic exchangeability, Theorem \ref{v67j9x2} yields
$$X\text{ exchangeable}\quad\Leftrightarrow\quad X\text{ stationary and condition \eqref{c55b9m} holds.}$$

We now give some characterizations of condition \eqref{c1}. Recall that $X_n$ {\em converges stably to} $\alpha$, where $\alpha$ is a r.p.m. on $\mathcal{B}$, if
$$P(X_n\in\cdot\mid H)\overset{weakly}\longrightarrow E\bigl\{\alpha(\cdot)\mid H\bigr\}\quad\quad\text{for all }H\in\mathcal{A}^+$$
where $\mathcal{A}^+=\bigl\{A\in\mathcal{A}:P(A)>0\bigr\}$. Equivalently,
$$E\bigl\{\alpha(f)\mid H\bigr\}=\lim_nE\bigl\{f(X_n)\mid H\bigr\}\quad\quad\text{for all }f\in C_b(S)\text{ and }H\in\mathcal{A}^+.$$

\begin{theorem}\label{f6yyh8}
Condition \eqref{c1} holds if and only if condition \eqref{x55hgf6j} holds, $X_n$ converges stably to $\alpha$, for some r.p.m. $\alpha$ on $\mathcal{B}$, and
\begin{gather}\label{mju889c4}
E\bigl\{\alpha(f)^2\bigr\}=\lim_nE\bigl\{\alpha_n(f)^2\bigr\}\quad\quad\text{for all }f\in C_b(S).
\end{gather}
\end{theorem}

\begin{proof}
First note that, if $H\in\mathcal{F}_k$ for some $k$, then
\begin{gather}\label{vg874n9i}
E\bigl\{1_H\,f(X_n)\bigr\}=E\bigl\{1_H\,E(f(X_n)\mid\mathcal{F}_{n-1})\bigr\}=E\bigl\{1_H\,\alpha_{n-1}(f)\bigr\}
\end{gather}
for all $f\in M_b(S)$ and $n>k$.

Assume condition \eqref{c1}. By Theorem \ref{bgy7m}, there is a r.p.m. $\alpha$ on $\mathcal{B}$ such that $\alpha$ is $\mathcal{T}$-measurable and $\alpha_n(f)\overset{P}\longrightarrow\alpha(f)$ for each $f\in M_b(S)$. Hence, condition \eqref{mju889c4} follows from the bounded convergence theorem. Next, fix $f\in C_b(S)$. If $H\in\bigcup_k\mathcal{F}_k$ and $P(H)>0$, one obtains
$$E\bigl\{\alpha(f)\mid H\bigr\}=\lim_nE\bigl\{\alpha_n(f)\mid H\bigr\}=\lim_nE\bigl\{f(X_n)\mid H\bigr\}$$
where the second equality depends on \eqref{vg874n9i}. Since $\alpha$ is $\mathcal{T}$-measurable and $\mathcal{T}\subset\sigma(X)$, by standard arguments, this implies
$$E\bigl\{\alpha(f)\mid H\bigr\}=\lim_nE\bigl\{f(X_n)\mid H\bigr\}\quad\quad\text{for all }H\in\mathcal{A}^+.$$
Hence, $X_n\rightarrow\alpha$ stably. Finally, condition \eqref{x55hgf6j} holds since \eqref{c1} $\Leftrightarrow$ \eqref{c55b9m}-\eqref{x55hgf6j}.

Conversely, assume conditions \eqref{x55hgf6j}-\eqref{mju889c4} and $X_n\rightarrow\alpha$ stably. Since \eqref{x55hgf6j} holds, to obtain \eqref{c1}, it suffices to prove condition \eqref{c55b9m}. Fix $f\in C_b(S)$. Since $X_n\rightarrow\alpha$ stably, using equation \eqref{vg874n9i} and arguing as above, one obtains
\begin{gather}\label{d56h8j}
E\bigl\{\alpha(f)\mid H\bigr\}=\lim_nE\bigl\{\alpha_n(f)\mid H\bigr\}\quad\quad\text{for all }H\in\mathcal{A}^+.
\end{gather}
Since $\alpha(f)$ is bounded and $\mathcal{A}$-measurable, $\alpha(f)$ is the uniform limit of a sequence $(U_k)$ of $\mathcal{A}$-simple functions. Since $U_k\rightarrow\alpha(f)$ uniformly, as $k\rightarrow\infty$, equation \eqref{d56h8j} implies
\begin{gather*}
E\bigl\{\alpha(f)^2\bigr\}=\lim_kE\bigl\{U_k\,\alpha(f)\bigr\}=\lim_k\lim_nE\bigl\{U_k\,\alpha_n(f)\bigr\}
\\=\lim_n\lim_kE\bigl\{U_k\,\alpha_n(f)\bigr\}=\lim_nE\bigl\{\alpha(f)\,\alpha_n(f)\bigr\}.
\end{gather*}
Hence, condition \eqref{mju889c4} yields
$$E\Bigl\{\bigl(\alpha_n(f)-\alpha(f)\bigr)^2\Bigr\}=E\bigl\{\alpha(f)^2\bigr\}+E\bigl\{\alpha_n(f)^2\bigr\}-2\,E\bigl\{\alpha_n(f)\,\alpha(f)\bigr\}\longrightarrow 0.$$
Therefore, $\alpha_n(f)\overset{P}\longrightarrow\alpha(f)$ and this concludes the proof.
\end{proof}

The next result deals with a special case. Recall that $\mu_n=(1/n)\,\sum_{i=1}^n\delta_{X_i}$ denotes the empirical measure.

\begin{theorem}\label{09ilmw3}
Suppose that, for each $f\in M_b(S)$,
\begin{gather}\label{x348uh6}
\mu_n(f)\quad\text{converges in probability to a degenerate random variable.}
\end{gather}
Then, condition \eqref{c1} is equivalent to condition \eqref{x55hgf6j} and
\begin{gather}\label{f56hj8m}
\lim_nE\bigl\{\alpha_n(f)^2\bigr\}=\lim_n E\bigl\{f(X_n)\bigr\}^2\quad\text{for each }f\in M_b(S).
\end{gather}
\end{theorem}

\begin{proof}
Assume condition \eqref{c1}. Then, condition \eqref{x55hgf6j} holds. Moreover, as noted in Remark (iii), there is a r.p.m. $\alpha$ such that $\alpha_n(f)\overset{P}\longrightarrow\alpha(f)$ and $\mu_n(f)\overset{P}\longrightarrow\alpha(f)$ for all $f\in M_b(S)$. Fix $f\in M_b(S)$. Since $\mu_n(f)\overset{P}\longrightarrow\alpha(f)$, condition \eqref{x348uh6} implies that $\alpha(f)$ is a degenerate random variable, so that
$$\alpha(f)=E\bigl\{\alpha(f)\bigr\}=\lim_nE\bigl\{\alpha_{n-1}(f)\bigr\}=\lim_nE\bigl\{f(X_n)\bigr\}\quad\quad\text{a.s.}$$
Therefore, Theorem \ref{f6yyh8} yields $\lim_nE\bigl\{f(X_n)\bigr\}^2=E\bigl\{\alpha(f)^2\bigr\}=\lim_nE\bigl\{\alpha_n(f)^2\bigr\}$.

Conversely, assume conditions \eqref{x55hgf6j} and \eqref{f56hj8m}. By \eqref{x55hgf6j} and the Vitali-Hahn-Saks theorem, $\lambda(\cdot):=\lim_nP(X_n\in\cdot)$ is a probability measure on $\mathcal{B}$ such that $\lambda(f)=\lim_nE\bigl\{f(X_n)\bigr\}$ for each $f\in M_b(S)$. Hence, condition \eqref{f56hj8m} yields
\begin{gather*}
E\Bigl\{\bigl(\alpha_n(f)-\lambda(f)\bigr)^2\Bigr\}=E\bigl\{\alpha_n(f)^2\bigr\}+\lambda(f)^2-2\,\lambda(f)\,E\bigl\{f(X_{n+1})\bigr\}\longrightarrow 0.
\end{gather*}
\end{proof}

Condition \eqref{x348uh6} is known to be true in various meaningful cases. For instance, it is true if $X$ is stationary and ergodic, or pairwise independent and identically distributed, or an irreducible positive recurrent Markov chain. More generally, condition \eqref{x348uh6} holds if $X$ satisfies the weak law of large numbers (see Remark (iii)) and $P(H)\in\{0,1\}$ for each $H$ in the $\sigma$-field generated by the random variables $\limsup_j\mu_{n_j}(f)$, where $f\in M_b(S)$ and $(n_j)$ is any subsequence. We also note that, in the special case where $X$ is pairwise independent, Theorem \ref{09ilmw3} is analogous to a characterization of convergence in probability of $X_n$ proved in \cite{EL}.

Next, to get condition \eqref{c1} via Theorem \ref{f6yyh8}, one should first identify the r.p.m. $\alpha$. In applications, this could be a shortcoming. Hence, we provide a version of Theorem \ref{f6yyh8} not involving $\alpha$ explicitly.

\begin{theorem}
Condition \eqref{c1} holds if and only if condition \eqref{x55hgf6j} holds and
\begin{gather}\label{vy78m}
\lim_nE\bigl\{f(X_n)\mid H\bigr\}\quad\text{exists,}\quad\mu_n(f)\text{ converges in probability,}
\\\text{and}\quad\lim_nE\Bigl\{\alpha_n(f)^2-\mu_n(f)^2\Bigr\}=0\quad\text{for all }f\in C_b(S)\text{ and }H\in\mathcal{A}^+.\notag
\end{gather}
\end{theorem}

\begin{proof} The ``only if" part is trivial. Conversely, assume conditions \eqref{x55hgf6j} and \eqref{vy78m}. Since $(S,\mathcal{B})$ is standard Borel and $\lim_nE\bigl\{f(X_n)\mid H\bigr\}$ exists for all $f\in C_b(S)$ and $H\in\mathcal{A}^+$, there is a r.p.m. $\alpha$ on $\mathcal{B}$ such that $X_n\rightarrow\alpha$ stably. By Corollary 2.4 of \cite{BPRSTOC}, since $\mu_n(f)$ converges in probability for all $f\in C_b(S)$, there is also a r.p.m. $\mu$ on $\mathcal{B}$ such that $\mu_n(f)\overset{P}\longrightarrow\mu(f)$ for each $f\in C_b(S)$. Moreover, by \eqref{vy78m},
$$E\bigl\{\mu(f)^2\bigr\}=\lim_nE\bigl\{\mu_n(f)^2\bigr\}=\lim_nE\bigl\{\alpha_n(f)^2\bigr\}\quad\quad\text{for all }f\in C_b(S).$$
Hence, because of Theorem \ref{f6yyh8}, it suffices to see that $\alpha(f)=\mu(f)$ a.s. for each $f\in C_b(S)$. Fix $f\in C_b(S)$ and recall that $\mu_n(f)-\frac{1}{n}\,\sum_{i=1}^{n}\alpha_{i-1}(f)\overset{a.s.}\longrightarrow 0$; see Remark (iii). For all $H\in\mathcal{A}^+$, one obtains
\begin{gather*}
E\bigl\{\alpha(f)\mid H\bigr\}=\lim_nE\bigl\{f(X_n)\mid H\bigr\}=\lim_nE\bigl\{\alpha_n(f)\mid H\bigr\}
\\=\lim_n\,\frac{1}{n}\,\sum_{i=1}^{n}E\bigl\{\alpha_{i-1}(f)\mid H\bigr\}=\lim_nE\bigl\{\mu_n(f)\mid H\bigr\}=E\bigl\{\mu(f)\mid H\bigr\}.
\end{gather*}
It follows that $\alpha(f)=\mu(f)$ a.s., and this concludes the proof.
\end{proof}

To close this section, we provide a proof of Theorem \ref{v67j9x2}.

\begin{proof}[Proof of Theorem \ref{v67j9x2}]
Under condition \eqref{c55b9m}, because of \cite[Cor. 2.4]{BPRSTOC}, there is a r.p.m. $\alpha$ on $\mathcal{B}$ such that $\alpha_n(f)\overset{P}\longrightarrow\alpha(f)$ for all $f\in C_b(S)$. Arguing as in the proof of Theorem \ref{f6yyh8}, one also obtains $X_n\rightarrow\alpha$ stably. Having noted these facts, fix $f,\,g\in C_b(S)$. Then,
\begin{gather*}
E\bigl\{f(X_{n+1})\,g(X_{n+2})\bigr\}=E\bigl\{f(X_{n+1})\,E(g(X_{n+2})\mid\mathcal{F}_{n+1})\bigr\}
\\=E\bigl\{f(X_{n+1})\,\alpha_{n+1}(g)\bigr\}=E\bigl\{f(X_{n+1})\,\alpha(g)\bigr\}+E\bigl\{f(X_{n+1})\,\bigl(\alpha_{n+1}(g)-\alpha(g)\bigr)\bigr\}.
\end{gather*}
Observe now that
$$\Abs{E\bigl\{f(X_{n+1})\,\bigl(\alpha_{n+1}(g)-\alpha(g)\bigr)\bigr\}}\le\sup\abs{f}\,E\Abs{\alpha_{n+1}(g)-\alpha(g)}\longrightarrow 0.$$
Moreover, since $X_n\rightarrow\alpha$ stably and $\alpha(g)$ is bounded and $\mathcal{A}$-measurable,
$$\lim_nE\bigl\{f(X_{n+1})\,\alpha(g)\bigr\}=E\bigl\{\alpha(f)\,\alpha(g)\bigr\}.$$
Therefore,
$$\lim_nE\bigl\{f(X_{n+1})\,g(X_{n+2})\bigr\}=E\bigl\{\alpha(f)\,\alpha(g)\bigr\}.$$
Proceeding in this way, by induction, one obtains
$$\lim_nE\left\{\prod_{j=1}^kf_j(X_{n+j})\right\}=E\left\{\prod_{j=1}^k\alpha(f_j)\right\}$$
for all $k\ge 1$ and all $f_1,\ldots,f_k\in C_b(S)$. This precisely means that $(X_{n+1},X_{n+2},\ldots)$ converges in distribution, as $n\rightarrow\infty$, to the exchangeable sequence $(Z_1,Z_2,\ldots)$ whose probability distribution is given by
$$P(Z_1\in B_1,\ldots,Z_k\in B_k)=E\left\{\prod_{j=1}^k\alpha(B_j)\right\}\quad\text{for all }k\ge 1\text{ and }B_1,\ldots,B_k\in\mathcal{B}.$$
\end{proof}

\section{Weak versions of conditional identity in distribution}\label{vh89f5}
Recall that $X$ is c.i.d. if $(X_1,\ldots,X_n,X_k)\sim (X_1,\ldots,X_n,X_{n+1})$ for $k>n\ge 0$. C.i.d. sequences have been introduced in \cite{BPRCID} and \cite{KAL} and then investigated or used in various papers. See e.g. \cite{BCL,BPREJP,BPR13,BERN21,BW2025,CZGV,FHW23,FPS} and references therein.

Condition \eqref{c2} holds whenever $X$ is c.i.d. In this case, in fact, $(\alpha_n(f):n\ge 0)$ is a uniformly bounded martingale for fixed $f\in M_b(S)$. Hence, $\alpha_n(f)$ converges a.s.

In this section, we focus on three weak versions of the c.i.d. condition. In the spirit of this paper, our main goal is to investigate whether they are still sufficient for \eqref{c1} or \eqref{c2}. However, apart from \eqref{c1} and \eqref{c2}, such versions could make some interest in themselves.

\subsection{Conditional identity in distribution of higher order} Let $m$ be a non-negative integer. Say that $X$ is $m$-c.i.d. if
$$E\bigl\{f(X_k)\mid\mathcal{F}_n\bigr\}=E\bigl\{f(X_{n+m+1})\mid\mathcal{F}_n\bigr\}\quad\text{a.s.}$$
for all $k>n+m$, all $n\ge 0$, and all $f\in M_b(S)$. If $m=0$, then $X$ is c.i.d. if and only if it is $0$-c.i.d. Moreover, if $X$ is $m$-c.i.d., it is $p$-c.i.d. for all integers $p>m$.

The $m$-c.i.d. condition looks reasonable in various problems. As a remarkable special case, $X$ is $m$-c.i.d. provided it is identically distributed and $m$-dependent. (Recall that $X$ is $m$-dependent if $(X_1,\ldots,X_n)$ is independent of $(X_i:i>n+m)$ for every $n$).

The $m$-c.i.d. condition has been included in this section for it is a natural weakening of the c.i.d. condition, which looks potentially useful in applications. However, condition \eqref{c1} may fail when $X$ is $m$-c.i.d. As shown in Example \ref{f56n9m2q}, it may be that $X$ is stationary, $m$-dependent but not asymptotically exchangeable. In this case, $X$ is $m$-c.i.d. (for it is identically distributed and $m$-dependent) but condition \eqref{c1} fails (for $X$ is not asymptotically exchangeable).

\subsection{Quasi-martingale predictive distributions} Let $(Y_n)$ be a sequence of real integrable random variables adapted to a filtration $(\mathcal{G}_n)$. Then, $(Y_n)$ is a quasi-martingale with respect to $(\mathcal{G}_n)$ if
$$\sum_nE\Abs{\,E(Y_{n+1}\mid\mathcal{G}_n)-Y_n}<\infty.$$
For instance, a non-negative supermartingale is a quasi-martingale (but not necessarily a martingale). Importantly, $Y_n$ converges a.s. provided it is a quasi-martingale such that $\sup_nE\abs{Y_n}<\infty$.

It is straightforward to check that $X$ is c.i.d. if and only if the sequence $(\alpha_n(f):n\ge 0)$ is a martingale, with respect to $(\mathcal{F}_n)$, for fixed $f\in M_b(S)$. Hence, trivially, a weakening of the c.i.d. condition which still implies \eqref{c2} is
\begin{gather}\label{qmc}
(\alpha_n(f):n\ge 0)\text{ is a quasi-martingale, with respect to }(\mathcal{F}_n),
\end{gather}
for each $f\in M_b(S)$. In fact, under \eqref{qmc}, $\alpha_n(f)$ converges a.s. since $(\alpha_n(f):n\ge 0)$ is a quasi-martingale such that $E\abs{\alpha_n(f)}\le\sup\abs{f}$ for all $n$. Note also that
$$\alpha_n(f)=E\bigl\{f(X_{n+1})\mid\mathcal{F}_n\bigr\}\quad\text{and}\quad E\bigl\{\alpha_{n+1}(f)\mid\mathcal{F}_n\bigr\}=E\bigl\{f(X_{n+2})\mid\mathcal{F}_n\bigr\}\quad\text{a.s.}$$
Therefore, condition \eqref{qmc} reduces to
$$\sum_nE\Abs{\,E\bigl\{f(X_{n+2})-f(X_{n+1})\mid\mathcal{F}_n\bigr\}}<\infty\quad\text{for each }f\in M_b(S).$$

Non-c.i.d. sequences satisfying condition \eqref{qmc} arise quite frequently in applications. We close this section with a few examples.

\begin{example}\textbf{(Recursive predictive distributions)}
If $(\Omega,\mathcal{A})=(S,\mathcal{B})$, a r.p.m. on $\mathcal{B}$ is said to be a {\em kernel} on $(S,\mathcal{B})$. For each $n\ge 0$, suppose
\begin{gather}\label{s45yh8m9}
\alpha_{n+1}(\cdot)=q_n(X_1,\ldots,X_n)\,\alpha_n(\cdot)\,+\,(1-q_n(X_1,\ldots,X_n))\,K_n(X_{n+1},\,\cdot)
\end{gather}
where $q_n:S^n\rightarrow (0,1)$ is a Borel function and $K_n$ a kernel on $(S,\mathcal{B})$. This type of predictive distributions have an intuitive interpretation and are quite popular in Bayesian predictive inference and species sampling; see e.g. \cite{BCL,BERN21,PIT}. Since
\begin{gather*}
E\Abs{\,E\bigl\{f(X_{n+2})-f(X_{n+1})\mid\mathcal{F}_n\bigr\}}\le E\bigl\{\abs{\alpha_{n+1}(f)-\alpha_n(f)}\bigr\}\le 2\,\sup\abs{f}\,E\Bigl\{1-q_n(X_1,\ldots,X_n)\Bigr\},
\end{gather*}
condition \eqref{qmc} trivially holds whenever
$$\sum_n\sup_{x\in S^n}\,\bigl(1-q_n(x)\bigr)<\infty.$$
Moreover, since the $K_n$ are still arbitrary, they can be chosen so that $X$ is not c.i.d.
\end{example}

\begin{example}\textbf{(Convex combinations of kernels)}
For each $n\ge 0$, choose a constant $d_n>0$ and a kernel $K_n$. Define $D_n=\sum_{i=0}^{n-1}d_i$ and
$$\alpha_n(\cdot)=(1/D_n)\,\sum_{i=0}^{n-1}d_i\,K_i(X_{i+1},\,\cdot)\quad\quad\text{where }n\ge 1.$$
Since \eqref{s45yh8m9} holds with $q_n=D_n/D_{n+1}$, a sufficient condition for \eqref{qmc} is
$$\sum_n\frac{d_n}{d_0+\ldots+d_n}<\infty.$$
\end{example}

\begin{example}\textbf{(Generalized P\'olya urns)}
An urn contains $b>0$ black balls and $r>0$ red balls. At each time $n\ge 1$, a ball is drawn and then replaced together with a random number of balls of the same color. Say that $B_n$ black balls or $R_n$ red balls are added to the urn according to whether $Y_n=1$ or $Y_n=0$, where $Y_n$ is the indicator of the event $\bigl\{$black ball at time $n\bigr\}$. To model this urn scheme, it is quite natural to let
$$P(Y_{n+1}=1\mid B_1,R_1,Y_1,\ldots,B_n,R_n,Y_n)=\frac{b+\sum_{i=1}^nB_iY_i}{b+r+\sum_{i=1}^n\bigl(B_iY_i+R_i(1-Y_i)\bigr)}.$$
Urns of this type have an history; see \cite{BCL,BPRCID,BCPR11,MF,PEM} and references therein. Now, define $X_n=(B_n,R_n,Y_n)$ and suppose:
\begin{itemize}

\item $(B_{n+1},R_{n+1})$ is independent of $(X_1,\ldots,X_n,Y_{n+1})$;

\item $0\le B_n,\,R_n\le b$ for some $b>0$ and $E(B_n)=E(R_n)$;

\item $\lim_nE(B_n)$, $\lim_nE(B_n^2)$, $\lim_nE(R_n^2)$ exist and $\lim_nE(B_n)>0$.

\end{itemize}
Then, as shown in \cite[Sect. 4.3]{BCPR11}, $P(Y_{n+1}=1\mid\mathcal{F}_n)$ is a quasi-martingale, so that it converges a.s. Thus, since $(B_{n+1},R_{n+1})$ is independent of $(X_1,\ldots,X_n,Y_{n+1})$, condition \eqref{c1} holds whenever $E\bigl\{f(B_n,R_n)\bigr\}$ approaches a limit for every bounded Borel $f$. This trivially happens, for instance,  if the sequence $(B_n,R_n)$ is identically distributed. However, $X$ is not necessarily c.i.d.
\end{example}

\subsection{Non-adapted c.i.d. sequences} Let $(\mathcal{G}_n)$ be a filtration on $(\Omega,\mathcal{A},P)$ satisfying $\mathcal{G}_n\subset\mathcal{F}_n$ for each $n\ge 0$. Define the r.p.m.'s
$$\beta_n(\cdot)=P(X_{n+1}\in\cdot\mid\mathcal{G}_n)$$
and note that $\beta_n(f)=E\bigl\{f(X_{n+1})\mid\mathcal{G}_n\bigr\}$ a.s. for all $f\in M_b(S)$.

In this section, we assume that
\begin{gather}\label{b65s2w}
E\bigl\{f(X_k)\mid\mathcal{G}_n\bigr\}=E\bigl\{f(X_{n+1})\mid\mathcal{G}_n\bigr\}\quad\text{a.s.}
\end{gather}
for all $k>n\ge 0$ and all $f\in M_b(S)$. Then, $X$ is identically distributed (since $\mathcal{G}_0=\{\emptyset,\Omega\}$) and $X$ is c.i.d. whenever $\mathcal{G}_n=\mathcal{F}_n$ for all $n$. More importantly, for fixed $f\in M_b(S)$, the sequence $(\beta_n(f):n\ge 0)$ is a uniformly bounded martingale with respect to $(\mathcal{G}_n)$. Hence, $\beta_n(f)$ converges a.s. In addition,
$$\beta_n(f)=E\bigl\{f(X_{n+1})\mid\mathcal{G}_n\bigr\}=E\Bigl\{E\bigl\{f(X_{n+1})\mid\mathcal{F}_n\bigr\}\mid\mathcal{G}_n\Bigr\}=E\bigl\{\alpha_n(f)\mid\mathcal{G}_n\bigr\}\text{ a.s.}$$
Therefore, a (natural) question is whether condition \eqref{c1} holds whenever $\mathcal{G}_n$ is close to $\mathcal{F}_n$ for large $n$. The next result is a possible answer to this question.

\begin{theorem}\label{g65cx3}
Suppose condition \eqref{b65s2w} holds and define $\mathcal{F}_\infty=\sigma\bigl(\cup_n\mathcal{F}_n\bigr)$ and $\mathcal{G}_\infty=\sigma\bigl(\cup_n\mathcal{G}_n\bigr)$. Then, a sufficient condition for \eqref{c1} is
\begin{gather}\label{ioc7}
\lim_nE\bigl\{\alpha_n(f)^2\bigr\}=\lim_nE\bigl\{\beta_n(f)^2\bigr\}\quad\quad\text{for each }f\in M_b(S).
\end{gather}
If $\mathcal{G}_\infty=\mathcal{F}_\infty$, condition \eqref{ioc7} is necessary as well.
\end{theorem}

\begin{proof}
Fix $f\in M_b(S)$. Under \eqref{ioc7},
\begin{gather*}
E\Bigl\{\bigl(\alpha_n(f)-\beta_n(f)\bigr)^2\Bigr\}=E\bigl\{\alpha_n(f)^2\bigr\}+E\bigl\{\beta_n(f)^2\bigr\}-2\,E\Bigl\{\beta_n(f)\,E\bigl\{\alpha_n(f)\mid\mathcal{G}_n\bigr\}\Bigr\}
\\=E\bigl\{\alpha_n(f)^2\bigr\}-E\bigl\{\beta_n(f)^2\bigr\}\longrightarrow 0.
\end{gather*}
Hence, $\alpha_n(f)$ converges in probability since $\alpha_n(f)-\beta_n(f)\overset{P}\longrightarrow 0$ and $\beta_n(f)$ converges a.s. (because of \eqref{b65s2w}). Next, suppose $\mathcal{G}_\infty=\mathcal{F}_\infty$ and condition \eqref{c1} holds. By Theorem \ref{bgy7m}, $\alpha_n(f)\overset{P}\longrightarrow\alpha(f)$ where the r.p.m. $\alpha$ is $\mathcal{T}$-measurable. By the martingale convergence theorem, $\beta_n(f)=E\bigl\{\alpha_n(f)\mid\mathcal{G}_n\bigr\}\overset{P}\longrightarrow E\bigl\{\alpha(f)\mid\mathcal{G}_\infty\bigr\}$; see e.g. Theorem 2 of \cite{BLAKDUB}. Since $\mathcal{T}\subset\mathcal{F}_\infty=\mathcal{G}_\infty$ and $\alpha$ is $\mathcal{T}$-measurable, then $E\bigl\{\alpha(f)\mid\mathcal{G}_\infty\bigr\}=\alpha(f)$ a.s. Hence, both $\alpha_n(f)$ and $\beta_n(f)$ converge in probability to $\alpha(f)$, so that $\lim_nE\bigl\{\alpha_n(f)^2\bigr\}=E\bigl\{\alpha(f)^2\bigr\}=\lim_nE\bigl\{\beta_n(f)^2\bigr\}$.
\end{proof}

In general, under \eqref{b65s2w}, condition \eqref{c1} could be expected to be true if $\mathcal{G}_n$ is close to $\mathcal{F}_n$ for large $n$. However, $\mathcal{G}_\infty=\mathcal{F}_\infty$ is not enough. Indeed, as shown by Example \ref{r45h7jb4}, condition \eqref{c1} may fail even if condition \eqref{b65s2w} holds and $\mathcal{G}_\infty=\mathcal{F}_\infty$.

Finally, apart from \eqref{c1}, condition \eqref{b65s2w} helps to get the weak law of large numbers.

\begin{theorem}\label{vc67m91q}
Let $g:[0,\infty)\rightarrow [0,\infty)$ be an increasing function such that $g(x)\le x$ for all $x\ge 0$ and $\lim_nn^{-2}\,\int_1^ng(x)\,dx=0$. Assume condition \eqref{b65s2w} and
$$\mathcal{F}_{n-[g(n)]}\subset\mathcal{G}_n\quad\quad\text{for each }n\ge 0$$
where $[\cdot]$ denotes the integer part. Then, $(1/n)\,\sum_{i=1}^{n}f(X_i)$ converges in probability for each $f\in M_b(S)$.
\end{theorem}

\begin{proof}
Let $f\in M_b(S)$. Since $\beta_n(f)$ converges a.s., it suffices to show that
$$\lim_n\frac{1}{n^2}\,E\Bigl\{\Bigl(\sum_{i=1}^n\bigl(f(X_i)-\beta_{i-1}(f)\bigr)\Bigr)^2\Bigr\}=0$$
or equivalently
$$\lim_n\frac{1}{n^2}\,\sum_{1\le i<j\le n}E\Bigl\{\bigl(f(X_i)-\beta_{i-1}(f)\bigr)\bigl(f(X_j)-\beta_{j-1}(f)\bigr)\Bigr\}=0.$$
For fixed $j$, since $\mathcal{F}_{j-1-[g(j-1)]}\subset\mathcal{G}_{j-1}$, one obtains
$$E\Bigl\{\bigl(f(X_i)-\beta_{i-1}(f)\bigr)\bigl(f(X_j)-\beta_{j-1}(f)\bigr)\Bigr\}=E\Bigl\{\bigl(f(X_i)-\beta_{i-1}(f)\bigr)\,E\Bigl(f(X_j)-\beta_{j-1}(f)\mid\mathcal{G}_{j-1}\Bigr)\Bigr\}=0$$
whenever $i<j-[g(j-1)]$. Therefore, it suffices noting that
\begin{gather*}\frac{1}{n^2}\,\sum_{j=2}^n\,\,\sum_{i=j-[g(j-1)]}^{j-1}E\Abs{\bigl(f(X_i)-\beta_{i-1}(f)\bigr)\bigl(f(X_j)-\beta_{j-1}(f)\bigr)}\le\frac{4\,\bigl(\sup\abs{f}\bigr)^2}{n^2}\,\sum_{j=2}^ng(j-1)
\\\le\frac{4\,\bigl(\sup\abs{f}\bigr)^2}{n^2}\,\int_1^ng(x)\,dx\longrightarrow 0.
\end{gather*}
\end{proof}

\section{Counterexamples}
All the counterexamples are collected in this final section.

\begin{example}\label{v778m4w}\textbf{$\bigl($Condition \eqref{c1} $\nRightarrow$ condition \eqref{c2}$\bigr)$}
The following example exhibits a situation where, not only \eqref{c1} holds and \eqref{c2} fails, but $X$ is also pairwise independent. Simpler examples are available if pairwise independence is not required. Suppose $(\Omega,\mathcal{A},P)$ is non-atomic and take the events $A_n,B_n,C_n\in\mathcal{A}$ such that:
\begin{itemize}

\item The sequence of triples $\bigl\{(A_n,B_n,C_n):n\ge 1\bigr\}$ is independent;

\item $A_n$ and $B_n$ are independent and $C_n=(A_n\cap B_n)\cup G_n$ with $G_n\subset A_n^c\cap B_n^c$;

\item $P(A_n)=P(B_n)=d_n$ and $P(G_n)=d_n-d_n^2$ where $0<d_n<1/2$.

\end{itemize}
Next, for each $n\ge 1$, define $X_{3n-2}=1_{A_n}$, $X_{3n-1}=1_{B_n}$ and $X_{3n}=1_{C_n}$. It is straightforward to check that each triple $(A_n,B_n,C_n)$ is pairwise independent. Hence, $X$ is pairwise independent. Since $X_{3n-2}$ is independent of $\mathcal{F}_{3n-3}$ and $X_{3n-1}$ is independent of $\mathcal{F}_{3n-2}$, one obtains
$$\alpha_{3n-3}(f)=\alpha_{3n-2}(f)=f(0)+(f(1)-f(0))\,d_n\quad\quad\text{a.s.}$$
Similarly,
\begin{gather*}
\alpha_{3n-1}(f)=E\bigl\{f(X_{3n})\mid X_{3n-2},X_{3n-1}\bigr\}
\\=f(0)+(f(1)-f(0))\,\Bigl\{X_{3n-2}X_{3n-1}+\frac{d_n}{1-d_n}\,(1-X_{3n-2})(1-X_{3n-1})\Bigr\}\quad\text{a.s.}
\end{gather*}
Thus, condition \eqref{c1} holds and condition \eqref{c2} fails whenever $d_n\rightarrow 0$ and $\sum_nd_n^2=\infty$. In fact, $\alpha_n(f)\overset{P}\longrightarrow f(0)$ since $d_n\rightarrow 0$. If $f(0)\neq f(1)$, however, $\alpha_{3n-1}(f)$ does not converge to $f(0)$ a.s., since the Borel-Cantelli lemma implies
$$P(X_{3n-2}=X_{3n-1}=1\text{ for infinitely many }n)=1.$$
\end{example}

\begin{example}\label{w34b7j}\textbf{$\bigl($Asymptotic exchangeability $\nRightarrow$ condition \eqref{c55b9m}$\bigr)$}
Let $(Y_n,Z_n)$ be an independent sequence of bivariate random variables. Suppose that $(Y_n,Z_n)$ has a density $f_n$, with respect to Lebesgue measure on $(0,1)^2$, where
$$f_n(y,z)=1+\sin(2\,\pi\,n\,y)\,\cos(2\,\pi\,z)\quad\text{ for all }(y,z)\in (0,1)^2.$$
Define $X_1=0$, $X_{2n}=Y_n$ and $X_{2n+1}=Z_n$ for all $n\ge 1$. Since the pairs $(Y_n,Z_n)$ are independent and $Y_n$ is uniformly distributed on $(0,1)$, letting $f(z)=\cos (2\,\pi\,z)$, one obtains
\begin{gather*}
\alpha_{2n-1}(f)=E\bigl\{f(Y_n)\mid Y_1,Z_1,\ldots,Y_{n-1},Z_{n-1}\bigr\}=E\bigl\{f(Y_n)\bigr\}=\int_0^1 \cos(2\,\pi\,z)\,dz=0
\end{gather*}
a.s. However,
\begin{gather*}
\alpha_{2n}(f)=E\bigl\{f(Z_n)\mid Y_1,Z_1,\ldots,Y_{n-1},Z_{n-1},Y_n\bigr\}
\\=E\bigl\{f(Z_n)\mid Y_n\bigr\}=\int_0^1 \cos(2\,\pi\,z)\,f_n(Y_n,z)\,dz=\frac{\sin(2\,\pi\,n\,Y_n)}{2}\quad\text{a.s.}
\end{gather*}
Hence, since $f\in C_b(S)$ and $\alpha_n(f)$ does not converge in probability, condition \eqref{c55b9m} fails. It remains to show that $X$ is asymptotically exchangeable. First note that
\begin{gather*}
\lim_nE\bigl\{g(Y_n)\,h(Z_n)\bigr\}=\lim_n\int_0^1\int_0^1g(y)\,h(z)\,f_n(y,z)\,dy\,dz
\\=\int_0^1g(y)\,dy\,\int_0^1h(z)\,dz\,+\,\lim_n\,\int_0^1h(z)\,cos (2\,\pi\,z)\,dz\,\int_0^1 g(y)\,\sin(2\,\pi\,n\,y)\,dy
\\=\int_0^1g(y)\,dy\,\int_0^1h(z)\,dz\quad\quad\text{for all }g,\,h\in C_b(S).
\end{gather*}
Based on this fact, for all $k\ge 1$ and $g_1,\ldots,g_k\in C_b(S)$, one obtains
\begin{gather*}
\lim_nE\left\{\prod_{i=1}^kg_i(X_{n+i})\right\}=\prod_{i=1}^k\int_0^1 g_i(t)\,dt.
\end{gather*}
This means that $(X_{n+1},X_{n+2},\ldots)\rightarrow (Z_1,Z_2,\ldots)$ in distribution, as $n\rightarrow\infty$, where $(Z_1,Z_2,\ldots)$ is i.i.d. with $Z_1$ uniformly distributed on $(0,1)$. Therefore, $X$ is asymptotically exchangeable (it is even asymptotically i.i.d.).
\end{example}

\begin{example}\label{f56n9m2q}\textbf{$\bigl($Stationary and $m$-dependent $\nRightarrow$ asymptotically exchangeable$\bigr)$}
Let $X_n=Y_n-Y_{n+m}$ where $m\ge 1$ and $(Y_n)$ is an i.i.d. sequence of real non-degenerate random variables. Then, $X$ is stationary and $m$-dependent. Let $\mathcal{T}=\bigcap_n\sigma(X_n,X_{n+1},\ldots)$. Since $\mathcal{T}\subset\bigcap_n\sigma(Y_n,Y_{n+1},\ldots)$ and $(Y_n)$ is i.i.d., then $P(A)\in\{0,1\}$ for each $A\in\mathcal{T}$. Toward a contradiction, suppose now that $X$ is asymptotically exchangeable. Since $X$ is stationary, this would imply exchangeability of $X$. In turn, exchangeability of $X$ and $P(A)\in\{0,1\}$ for each $A\in\mathcal{T}$, would imply that $X$ is i.i.d. But $X$ is not i.i.d. Hence, $X$ is not asymptotically exchangeable.
\end{example}

\begin{example}\label{r45h7jb4}\textbf{$\bigl($Condition \eqref{b65s2w} and $\mathcal{G}_\infty=\mathcal{F}_\infty$ $\nRightarrow$ condition \eqref{c1}$\bigr)$}
Take $X$ as in Example \ref{f56n9m2q} and define $\mathcal{G}_n=\bigl\{\emptyset,\Omega\bigr\}$ for $n\le m$ and $\mathcal{G}_n=\mathcal{F}_{n-m}$ for $n>m$. Then, $\mathcal{G}_\infty=\mathcal{F}_\infty$ and condition \eqref{c1} fails (due to $X$ is not asymptotically exchangeable). Moreover, since $X$ is $m$-dependent, $E\bigl\{f(X_k)\mid\mathcal{G}_n\bigr\}=E\bigl\{f(X_1)\bigr\}$ a.s. for all $k>n\ge 0$ and $f\in M_b(S)$. Thus, condition \eqref{b65s2w} holds.
\end{example}

To introduce our last example, recall that $E\bigl\{f(X_{n+2})\mid\mathcal{F}_n\bigr\}=E\bigl\{\alpha_{n+1}(f)\mid\mathcal{F}_n\bigr\}$ a.s. Hence, a necessary condition for \eqref{c1} is
\begin{gather}\label{vg7uj9}
E\bigl\{f(X_{n+2})-f(X_{n+1})\mid\mathcal{F}_n\bigr\}\overset{P}\longrightarrow 0\quad\quad\text{for each }f\in M_b(S).
\end{gather}
However, \eqref{vg7uj9} does not suffice for \eqref{c1} even if various other conditions are satisfied. Interestingly, this may happen even in the classical CLT.

\begin{example}\textbf{$\bigl($Condition \eqref{c1} may fail even if condition \eqref{vg7uj9} holds, $X_n$ converges stably, and $X$ is asymptotically exchangeable$\bigr)$} Let $W,Z_1,Z_2,\ldots$ be real random variables on $(\Omega,\mathcal{A},P)$. Suppose $W\sim N(0,1)$ and the $Z_n$ are i.i.d. with $E(Z_1)=0$ and $E(Z_1^2)=1$. Then, it is well known that
$$X_n:=n^{-1/2}\sum_{i=1}^nZ_i\longrightarrow N(0,1)\quad\quad\text{stably.}$$
Next, for any $k\ge 1$, the vector $(X_{n+1},\ldots,X_{n+k})$ can be written as
$$\left(X_{n+1},\,\sqrt{\frac{n+1}{n+2}}\,X_{n+1},\dots,\sqrt{\frac{n+1}{n+k}}X_{n+1}\right)\,+\,\left(0,\,\frac{Z_{n+2}}{\sqrt{n+2}},\ldots,\frac{\sum_{i=n+2}^{n+k}Z_i}{\sqrt{n+k}}\right).$$
Hence, $(X_{n+1},\ldots,X_{n+k})\rightarrow (W,\ldots,W)$ in distribution for every $k\ge 1$, or equivalently $(X_{n+1},X_{n+2},\ldots)\longrightarrow (W,W,\ldots)$ in distribution. Thus, $X$ is asymptotically exchangeable. Next, to prove \eqref{vg7uj9}, we also assume that the probability distribution of $Z_1$ is absolutely continuous with respect to Lebesgue measure, so that
$$\sup_{A\in\mathcal{B}}\,\abs{P(X_n\in A)-P(W\in A)}\rightarrow 0.$$
We also note that $X_{n+2}-X_{n+1}\overset{a.s.}\longrightarrow 0$, and thus $g(X_{n+2})-g(X_{n+1})\overset{a.s.}\longrightarrow 0$ for each Lipschitz function $g$. Fix $f\in M_b(S)$ and $\epsilon>0$. Since $X_n$ converges in total variation, there is a bounded Lipschitz function $g$ such that
$E\abs{f(X_n)-g(X_n)}<\epsilon$ for large $n$. Therefore,
\begin{gather*}
E\abs{f(X_{n+2})-f(X_{n+1})}\le E\abs{f(X_{n+2})-g(X_{n+2})}+E\abs{g(X_{n+2})-g(X_{n+1})}\,+
\\+\,E\abs{g(X_{n+1})-f(X_{n+1})}< 3\,\epsilon\quad\quad\text{for large }n.
\end{gather*}
This proves $\lim_n E\abs{f(X_{n+2})-f(X_{n+1})}=0$, which in turn implies condition \eqref{vg7uj9}. Finally, we prove that condition \eqref{c1} fails. Since $X_n$ does not converge in probability, there is a bounded Lipschitz function $g$ such that $g(X_n)$ does not converge in probability. Let $\nu$ denote the probability distribution of $Z_1$. Since
$$\alpha_n(g)=E\bigl\{g(X_{n+1})\mid\mathcal{F}_n\bigr\}=\int g\left(\frac{\sqrt{n}}{\sqrt{n+1}}\,X_n+\frac{x}{\sqrt{n+1}}\right)\,\nu(dx)\quad\text{a.s.,}$$
one obtains
\begin{gather*}\Abs{\alpha_n(g)-g(X_n)}\le\int\Abs{g\left(\frac{\sqrt{n}}{\sqrt{n+1}}\,X_n+\frac{x}{\sqrt{n+1}}\right)-g(X_n)}\,\nu(dx)\le c\,\,\frac{\abs{X_n}+E\abs{Z_1}}{\sqrt{n+1}}\overset{P}\longrightarrow 0
\end{gather*}
where $c$ is the Lipschitz constant of $g$. Hence, condition \eqref{c1} fails. In fact, $\alpha_n(g)$ does not converge in probability since $g(X_n)$ does not converge in probability.
\end{example}


\begin{thebibliography}{99}

\bibitem{ALDOUS} Aldous D.J. (1985) Exchangeability and related topics, {\em Ecole d’eté de Probabilités de Saint-
Flour XIII, Lecture Notes in Math.} 1117, Springer, Berlin.

\bibitem{BCL} Bassetti F., Crimaldi I., Leisen F. (2010) Conditionally identically distributed species sampling sequences, {\em Adv. Appl. Probab.}, 42, 433-459.

\bibitem{BPRCID} Berti P., Pratelli L., Rigo P. (2004) Limit theorems for a class of identically distributed random variables, {\em Ann. Probab.}, 32, 2029-2052.

\bibitem{BPRSTOC} Berti P., Pratelli L., Rigo P. (2006) Almost sure weak convergence of random probability measures, {\em Stochastics}, 78, 91-97.

\bibitem{BCPR11} Berti P., Crimaldi I., Pratelli L., Rigo P. (2011) A central limit theorem and its applications to multicolor randomly reinforced urns, {\em J. Appl. Probab.}, 48, 527-546.

\bibitem{BPREJP} Berti P., Pratelli L., Rigo P. (2012) Limit theorems for empirical processes based on dependent data, {\em Electr. J. Probab.}, 17, 1-18.

\bibitem{BPR13} Berti P., Pratelli L., Rigo P. (2013) Exchangeable sequences driven by an absolutely continuous random measure, {\em Ann. Probab}, 41, 2090-2102.

\bibitem{BERN21} Berti P., Dreassi E., Pratelli L., Rigo P. (2021) A class of models for Bayesian predictive inference, {\em Bernoulli}, 27, 702-726.

\bibitem{BW2025} Bissiri P.G., Walker S.G. (2025) Bayesian analysis with conditionally identically distributed sequences, {\em Electr. J. Statistics}, 19, 1609-1632.

\bibitem{BLAKDUB} Blackwell D., Dubins L.E. (1962) Merging of opinions with increasing information, {\em Ann. Math. Statist.}, 33, 882-886.

\bibitem{CZGV} Cassese A., Zhu W., Guindani M., Vannucci M. (2019) A Bayesian nonparametric spiked process prior for dynamic model selection, {\em Bayesian Anal.}, 14, 553-572.

\bibitem{DV} Dawid A.P., Vovk V.G. (1999) Prequential probability: principles and properties, {\em Bernoulli}, 5, 125-162.

\bibitem{EL} Etemadi N., Lenzhen A. (2004) Convergence of sequences of pairwise independent random variables, {\em Proc. Amer. Math. Soc.}, 132, 1201-1202.

\bibitem{FHW23} Fong E., Holmes C., Walker S.G. (2023) Martingale posterior distributions (with discussion), {\em J. Royal Stat. Soc.} B, 85, 1357-1391.

\bibitem{FY} Fong E., Yiu A. (2024) Asymptotics for parametric martingale posteriors, {\em arXiv:2410.17692v1}

\bibitem{FPS} Fortini S., Petrone S., Sporysheva P. (2018) On a notion of partially conditionally identically distributed sequences, {\em Stoch. Proc. Appl.}, 128, 819-846.

\bibitem{HMW} Hahn P.R., Martin R., Walker S.G. (2018) On recursive Bayesian predictive distributions, {\em J.A.S.A.}, 113, 1085-1093.

\bibitem{KAL} Kallenberg O. (1988) Spreading and predictable sampling in exchangeable sequences and processes, {\em Ann. Probab.}, 16, 508-534.

\bibitem{MF} May C., Flournoy N. (2009) Asymptotics in response-adaptive designs generated by a two-color, randomly reinforced urn, {\em Ann. Statist.}, 37, 1058-1078.

\bibitem{MOW} Moya B., Walker S.G. (2025) Martingale posterior distributions for time series models, {\em Statist. Sci.}, 40, 68-80.

\bibitem{PEM} Pemantle R. (2007) A survey of random processes with reinforcement, {\em Prob. Surveys}, 4, 1-79.

\bibitem{PIT} Pitman J. (1996) Some developments of the Blackwell-MacQueen urn scheme, {\em Statistics, Probability and Game Theory, IMS Lect. Notes Mon. Series}, 30, 245-267.

\end{thebibliography}
\end{document}